 \documentclass[a4paper,13pt,abstracton]{article}
 \usepackage[utf8]{inputenc}
   \providecommand{\keywords}[1]{\textbf{\textit{Key words:}} #1}
   \usepackage{amsmath,amsthm,verbatim,amssymb,amsfonts,amscd, graphicx}
 \usepackage{geometry}
 \numberwithin{equation}{section}
 \newtheorem{thm}{Theorem}[section]
 \newtheorem{lemma}[thm]{Lemma}
 \newtheorem{definition}[thm]{Definition}
 \newtheorem{cor}[thm]{Corollary}
 \newtheorem{prop}[thm]{Proposition}
 \newtheorem{rem}[thm]{Remark}
 \newtheorem{conj}[thm]{Conjecture}
 \newtheorem{claim}[thm]{Claim}
 \usepackage{enumerate}
 \usepackage[square,comma,super,numbers,sort&compress]{natbib}
  \usepackage{hyperref}
 \hypersetup{colorlinks=true, urlcolor=Cerulean, citecolor=red}
 \usepackage[T1]{fontenc}
 \usepackage{nameref,cleveref}
 \usepackage{nicefrac,xfrac}
 \usepackage[all]{xy}
 \usepackage{color}

\begin{document}
\title{\textbf {Vector Bundles on Flag varieties}}

\author{Rong Du \thanks{School of Mathematical Sciences
Shanghai Key Laboratory of PMMP,
East China Normal University,
Rm. 312, Math. Bldg, No. 500, Dongchuan Road,
Shanghai, 200241, P. R. China,
rdu@math.ecnu.edu.cn.
The Research is Sponsored by National Natural Science Foundation of China (Grant No. 11531007), Natural Science Foundation of China and the Israel Science Foundation (Grant No. 11761141005) and Science and Technology Commission of Shanghai Municipality (Grant No. 18dz2271000).},
Xinyi Fang
\thanks{School of Mathematical Sciences
Shanghai Key Laboratory of PMMP,
East China Normal University,
No. 500, Dongchuan Road,
Shanghai, 200241, P. R. China,
2315885681@qq.com.
The Research is Sponsored by National Natural Science Foundation of China (Grant No. 11531007) and Science and Technology Commission of Shanghai Municipality (Grant No. 18dz2271000).}
and Yun Gao
\thanks{School of Mathematical Sciences,
Shanghai Jiao Tong University,
Shanghai 200240, P. R. of China,
gaoyunmath@sjtu.edu.cn.
The Research is Sponsored by National Natural Science Foundation of China (Grant No. 11531007).
}
}

\date{}
\maketitle

\begin{center}
{In memory of our friend Prof. Yi Zhang}
\end{center}

\begin{abstract}
We study vector bundles on flag varieties over an algebraically closed field $k$. In the first part, we suppose $G=G_k(d,n)$ $(2\le d\leq n-d)$ to be the Grassmannian manifold parameterizing linear subspaces of dimension $d$ in $k^n$, where $k$ is an algebraically closed field of characteristic $p>0$. Let $E$ be a uniform vector bundle over $G$ of rank $r\le d$. We show that $E$ is either a direct sum of line bundles or a twist of a pull back of the universal bundle $H_d$ or its dual $H_d^{\vee}$ by a series of absolute Frobenius maps. In the second part, splitting properties of vector bundles on general flag varieties $F(d_1,\cdots,d_s)$ in characteristic zero are considered. We prove a structure theorem for bundles over flag varieties which are uniform with respect to the $i$-th component of the manifold of lines in $F(d_1,\cdots,d_s)$. Furthermore, we generalize the Grauert-M$\ddot{\text{u}}$lich-Barth theorem to flag varieties. As a corollary, we show that any strongly uniform $i$-semistable $(1\le i\le n-1)$ bundle over the complete flag variety splits as a direct sum of special line bundles.
\end{abstract}

\section{Introduction}
It is classically known that any vector bundle on a projective line over arbitrary algebraically closed field $k$ splits as a direct sum of line bundles. However, if the dimension of a projective space is bigger than or equal to two, the situation is pretty involved. So the splitting of vector bundles on higher dimensional projective spaces has long been a major concern among the problems on vector bundles in algebraic geometry. For non-splitting vector bundles, there are many classification results on some special classes of vector bundles. One of the classes that has been studied more widely is \emph{uniform} vector bundles; that is, bundles whose splitting type is independent of the chosen line. The notion of a uniform vector bundle appears first in a paper of Schwarzenberger\cite{ref1}. In characteristic zero, much work has been done on the classification of uniform vector bundles over projective spaces. In 1972, Van de Ven \cite{ref2} proved that for $n>2$, the uniform 2-bundles over $\mathbb{P}_k^n$ split and the uniform 2-bundles over $\mathbb{P}_k^2$ are precisely the bundles $\mathcal{O}_{\mathbb{P}_k^2}(a)\bigoplus\mathcal{O}_{\mathbb{P}_k^2}(b)$
and $T_{\mathbb{P}_k^2}(a),a,b\in\mathbb{Z}$. In 1976, Sato \cite{ref3} proved that for $2<r<n$, the uniform $r$-bundles over $\mathbb{P}_k^n$ split by using a theorem of Tango \cite{ref16} about holomorphic mappings from projective spaces to Grassmannians. In 1978, Elencwajg \cite{ref4} extended the investigations of Van de Ven to show that uniform vector bundles of rank 3 over $\mathbb{P}_k^2$  are of the form
$$\mathcal{O}_{\mathbb{P}_k^2}(a)\bigoplus\mathcal{O}_{\mathbb{P}_k^2}(b)\bigoplus\mathcal{O}_{\mathbb{P}_k^2}(c), ~ T_{\mathbb{P}_k^2}(a)\bigoplus\mathcal{O}_{\mathbb{P}_k^2}(b)\quad\text{or}\quad S^2T_{\mathbb{P}_k^2}(a),$$ where $a$, $b$, $c\in \mathbb{Z}$.  Sato \cite{ref3} had previously shown that for $n$ odd, uniform $n$-bundles over $\mathbb{P}_k^n$ are of the forms
 $$\oplus_{i=1}^{n}\mathcal{O}_{\mathbb{P}_k^n}(a_i),~ T_{\mathbb{P}_k^n}(a)\quad\text{or}\quad \Omega^1_{\mathbb{P}_k^n}(b),$$  where $a_i,a, b\in \mathbb{Z}$. So the results of Elencwajg and Sato yield a complete classification of uniform 3-bundles over $\mathbb{P}_k^n$. In particular all uniform 3-bundles over $\mathbb{P}_k^n$ are homogeneous. Later, Elencwajg, Hirschowitz and Schneider \cite{ref5} showed that Sato's result is also true for $n$ even. Around 1982, Ellia \cite{ref7} and Ballico \cite{ref8} independently proved that for $n\ge 3$, the uniform $(n+1)$-bundles over $\mathbb{P}_k^{n}$ are of the form
  $$\oplus_{i=1}^{n+1}\mathcal{O}_{\mathbb{P}_k^n}(a_i),~ T_{\mathbb{P}_k^n}(a)\bigoplus\mathcal{O}_{\mathbb{P}_k^n}(b)\quad\text{or}\quad\Omega^1_{\mathbb{P}_k^n}(c)\bigoplus\mathcal{O}_{\mathbb{P}_k^n}(d), $$ where $a_i, a,b,c,d\in\mathbb{Z}$. One can go over the good reference by Okonek, Schneider and Spindler \cite{ref14} for related topics. Later, similar results have been extensively studied for uniform vector bundles on varieties swept out by lines, such as quadrics \cite{ref17, ref18}, Grassmannians\cite{ref15} and special Fano manifolds \cite{ref9}.

In positive characteristic, for $2\le r<n$, the uniform $r$-bundles over $\mathbb{P}_k^n$ split by Sato's result\cite{ref3}. The classification problem of uniform $n$-bundles on $\mathbb{P}_k^n$ has been solved by Lange \cite{ref10} for $n=2$ and Ein \cite{ref6} for all $n$. It seems that the classification of uniform vector bundles over other projective manifolds covered by lines in characteristic $p$ is still open. In the first part of the paper, we consider the uniform vector bundles on Grassmannians in positive characteristic and prove the following main theorem.

\begin{thm}\emph{(Theorem \ref{b} and Theorem \ref{z})}\label{m}
	Let $G=G_k(d,n)$ $(d\leq n-d)$ be the Grassmannian manifold parameterizing linear subspaces of dimension $d$ in $k^n$, where $k$ is an algebraically closed field of characteristic $p>0$. Let $E$  be a uniform vector bundle over $G$ of rank $r\le d$.
\begin{itemize}
\item If $r<d$, then $E$ is a direct sum of line bundles.
\item If $r=d$, then $E$ is either a direct sum of line bundles or a twist of a pull back of the universal bundle $H_d$ or its dual $H_d^{\vee}$ by a series of absolute Frobenius maps.
\end{itemize}
\end{thm}

\begin{rem}
\emph{The first part of the theorem holds for any algebraically closed field. The result in characteristic zero is due to Guyot \cite{ref15}. We generalize the method of Elenwajg-Hirschowitz-Schneider \cite{ref5}, by which the authors considered the case over projective spaces, to deal with Grassmannians over any algebraically closed field. Our basic idea follows from the strategy in  \cite{ref5} but the complexity in the case of Grassmannians is far better than in the case of projective spaces. For example, Proposition \ref{e} cannot be generalized from the case for projective spaces directly. We need to strengthen the condition and use different. However, for $r<d$, our method is independent of the characteristic of the field.
For the case $r=d$, we checked that the first part of Guyot's paper \cite{ref15}, which is independent of the characteristic of the field $k$, is for studying the Chow ring and axiomatically computing Chern class of vector bundles of flag varieties. However, Guyot's method can only handle the case for $b=-1$, i.e. $(0,\ldots,0,-1)$ in the proof of Theorem \ref{m} because this is the only case for characteristic zero. In characteristic $p$, we need to use Katz's Lemma.
So for the second part of the theorem, we mainly use Ein's \cite{ref6} ideas for projective spaces and Katz's \cite{ref11} key lemma to study vector bundles in positive characteristic.	}
\end{rem}

The proof of the first part of the above theorem has a surprising consequence for vector bundles of arbitrary rank.
	\begin{cor}\emph{(Corollary \ref{d})}
Let $E$ be a vector bundle over $G=G(d,n)$ $(d\ge 2)$ over an algebraically closed field. If $E$ splits as a direct sum of line bundles when it restricts to every $\mathbb{P}_k^{d}\subseteq G$, then $E$ splits as a direct sum of line bundles over $G$.
	\end{cor}
	
In the second part of the paper, we consider vector bundles over flag varieties in characteristic zero.
For fixed integer $n$, let $F:=F(d_1,\ldots,d_s)$ be the flag manifold parameterizing flags \[V_{d_1}\subseteq\cdots\subseteq V_{d_s}\subseteq k^n\]
		 where $dim(V_{d_i})=d_i, 1\le i \le s$. Let $F^{(i)}:=F(d_1,\ldots,d_{i-1},d_i-1,d_i+1,d_{i+1},\ldots,d_s)$ be the $i$-th irreducible component of the manifold of lines in $F$. (From now we specify that if the two adjacent integers in the expression of flag varieties such as $F^{(i)}$ are equal, we keep only one of them. Please see Section \ref{flag} for the notations.)

We separate our discussion into two cases:

Case I: $d_i-1=d_{i-1}$ and $d_i+1=d_{i+1}$, then we have the natural projection $F\rightarrow F^{(i)}$;

Case II: $d_i-1\neq d_{i-1}$ or $d_i+1\neq d_{i+1}$, then we have the \emph{standard diagram}
\begin{align}
\xymatrix{
	F(d_1,\ldots,d_{i-1},d_i-1,d_i,d_i+1,d_{i+1},\ldots d_s)\ar[d]^{q_1}   \ar[r]^-{q_2} & F^{(i)}\\
	F=F(d_1,\ldots,d_s).
}
\end{align}
\begin{thm}\emph{(Theorem \ref{main})}
Fix integer $i, 1\le i \le s$. Let $E$ be an algebraic $r$-bundle over $F$ of type $\underline{a}_E^{(i)}=(a_1^{(i)},\ldots,a_r^{(i)}), a_1^{(i)}\geq\cdots\geq a_r^{(i)}$ with respect to $F^{(i)}$. If for some $t<r$,
 \[
	a_t^{(i)}-a_{t+1}^{(i)}\geq
	\left\{
	\begin{array}{ll}
	1, & and~F^{(i)}~ \text{is in Case I}\\
	2, & and ~F^{(i)}~ \text{is in Case II},
	\end{array}
	\right.\]
then there is a normal subsheaf $K\subseteq E$ of rank $t$ with the following properties: over the open set $V_E^{(i)}=q_1({q_2}^{-1}(U_E^{(i)}))\subseteq F$, the sheaf $K$ is a subbundle of $E$, which on the line $L\subseteq F$ given by $l\in U_E^{(i)}$ has the form\[
K|L\cong\oplus_{j=1}^{t}\mathcal{O}_L(a_j^{(i)}).
\]
\end{thm}

\begin{rem}
\emph{Please see Definition \ref{UEi} for the notation $U_E^{(i)}$. A coherent sheaf $\mathcal{F}$ is said to be "normal" in the sense of Barth (\cite{ref21}, p.128) if the restriction $\mathcal{F}(U)\rightarrow\mathcal{F}(U\backslash A)$ is bijective for every open subset $U$ and a closed subset $A$ of $U$ of codimension at least $2$ (cf. Definition \ref{normal}).}
\end{rem}
Using the above theorem, we generalize the Grauert-M$\ddot{\text{u}}$lich-Barth theorem (cf. \cite{ref21}, Theorem 1) to flag varieties as follows.

\begin{cor}\emph{(Corollary \ref{gap})}
Fix integer $i, 1\le i \le s$, for an i-semistable $r$-bundle $E$ over $F$ of type $\underline{a}_E^{(i)}=(a_1^{(i)},\ldots,a_r^{(i)}), a_1^{(i)}\geq\cdots\geq a_r^{(i)}$ with respect to $F^{(i)}$, we have \[
a_j^{(i)}-a_{j+1}^{(i)}\leq 1~~ \text{for all}~ j=1,\ldots, r-1.
\]
In particular, if $F^{(i)}$ is in Case I, then we have $a_j^{(i)}$'s are constant for all $1\leq j\leq r$.
\end{cor}

\begin{rem}
\emph{A generalization of the Grauert-M$\ddot{\text{u}}$lich-Barth theorem to normal projective varieties in characteristic zero is proven in Huybrechts and Lehn's famous book (see \cite{ref24} Theorem 3.1.2). But the definition of our semistability or just slope is a little bit different from the usual ones. The canonical way is due to Mumford and Takemoto who embed the variety into a projective space and fix an ample divisor to define the degree and the slope of vector bundles. Another way so called Gieseker-semistability is by replacing the degree in the Mumford-Takemoto's definition with Hilbert polynomial to define the slope (see for example \cite{ref24} Definition 1.2.3). However these two ways both depend on the embeddings. Our definition of semistability for the flag varieties (see Definition \ref{i}) is intrinsic because the flag varieties are uniruled. If one embeds the flag variety into a projective space and takes ample divisors whose intersection is a single line on the flag variety, then our definition is the same as Mumford-Takemoto's definition.}
\end{rem}

\begin{definition}
If there exists some integer $i$ such that the splitting type of a vector bundle $E$ is uniform for any line $L\subseteq F$ given by $l\in F^{(i)}$, then $E$ is called uniform vector bundle with respect to $F^{(i)}$. $E$ is called strongly uniform on $F$ if the splitting type of $E$ is uniform for any line $L\subseteq F$.
\end{definition}

\begin{cor}\emph{(Corollary \ref{x})}
If $E$ is a strongly uniform $i$-semistable $(1\le i\le n-1)$ $r$-bundle over the complete flag $F$, then $E$ splits as a direct sum of line bundles. In addition $E|L\cong \mathcal{O}_L(a)^{\oplus r}$ for every line $L\subseteq F$, where $a\in \mathbb{Z}$.
\end{cor}

\section{Preliminaries}
Denote by $G$ the Grassmannian $G_k(d, n)$ of $d$-dimensional linear
subspaces in $V=k^n$, where $k$ is an algebraically closed field. Of course we may also consider $G_k(d, n)$ in its
projective guise as $\mathbb{G}_k(d-1, n-1)$, the Grassmannian of
projective $(d-1)$-planes in $\mathbb{P}_k^{n-1}$.

Let $\mathcal{V}:=G\times V$ be the trivial vector bundle of rank $n$ on $G$ whose fiber
at every point is the vector space $V$. We write $H_d$ for the $d$-subbundle of $\mathcal{V}$ whose
fiber at a point $[\Lambda]\in G$ is the subspace $\Lambda$ itself; that is,
\[(H_d)_{[\Lambda]}=\Lambda\subseteq V=\mathcal{V}_{[\Lambda]}.\]
$H_d$ is called the \emph{universal subbundle} on $G$; the rank $n-d$ quotient bundle $Q_{n-d}=\mathcal{V}/H_d$ is called the \emph{universal quotient bundle}, i.e.,
\begin{equation}\label{tau}
0\rightarrow H_d\rightarrow \mathcal{V}\rightarrow Q_{n-d}\rightarrow 0.
\end{equation}
We write line bundle $\mathcal{O}_G(1)$, which is called the \emph{Pl$\ddot{u}$cker bundle} on $G$, to be the pull back of $\mathcal{O}_{\mathcal{P}}(1)$ under the \emph{Pl$\ddot{u}$cker embedding } \[
\iota: G_k(d,n)\rightarrow\mathcal{P}=\mathbb{P}(\bigwedge^{d}k^n).
\]

\begin{definition}
Let $F:=F(d_1,\ldots,d_s)$ be the flag manifold parameterizing flags \[V_{d_1}\subseteq\cdots\subseteq V_{d_s}\subseteq k^n\]
where $dim(V_{d_i})=d_i, 1\le i \le s$. In particular, flag manifold $F(1, \ldots, n-1)$ is called the complete flag manifold.
\end{definition}
Given a flag $V_{d-1}\subseteq V_{d+1}\subseteq k^n$, the set of $d$-dimensional subspaces $W\subseteq k^n$ such that\[
V_{d-1}\subseteq W\subseteq V_{d+1}
\]
is the projectivization of the quotient space $V_{d+1}/V_{d-1}$, so the set is isomorphic to $\mathbb{P}_k^1$. It follows that the flag manifold $F(d-1,d+1)$ be the manifold of lines in $G$ and the flag manifold $F(d-1,d,d+1)$ can be written as\[
F(d-1,d,d+1)=\{(x,L)\in G\times F(d-1,d+1) |x\in L\}.
\]
Let $\bar{F}:=F(d-1,d,d+1)$, $F_1:=F(d-1,d)$ and $F_2:=F(d,d+1)$. Then we have the following two diagrams:
\begin{align}
	\xymatrix{
		F_1   \ar[rrdd]_{q_{11}} && \bar{F}\ar[dd]^{q_1}\ar[ll]_{pr_1} \ar[rr]^{pr_2}  && F_2 \ar[lldd]^{q_{12}} \\
		\\
		&&G
	}
\end{align}
and
\begin{align}
	\xymatrix{
		\bar{F}\ar[d]^{q_1}   \ar[r]^-{q_2} & F(d-1,d+1),\\
		G
	}
\end{align}
where all morphisms in the above diagrams are projections.

\begin{rem}\label{g}\emph{(\cite{ref15} Section~ 2.2)}
\emph{The mapping $q_{11}~(resp.~q_{12})$ identifies $F_1~(resp.~F_2)$ with the projective bundle $\mathbb{P}(H_{d})~(resp.~\mathbb{P}(Q_{n-d}^{\vee}))$ of $G$.
Let $\mathcal{H}_{H_{d}^{\vee}}~(resp.~  \mathcal{H}_{Q_{n-d}})$ be the tautological line bundle on $\bar{F}$ associated to $F_1~(resp.~ F_2)$, i.e. \[{pr_1}_{*}\mathcal{H}_{H_{d}^{\vee}}=\mathcal{O}_{F_1}(-1)~(resp.~ {pr_2}_{*}\mathcal{H}_{Q_{n-d}}=\mathcal{O}_{F_2}(-1)).\]}
\end{rem}
\begin{definition}
Let $X$ be a noetherian scheme in characteristic $p$ and $E$ be a locally free coherent sheaf on $X$. We define the associated projective space bundle $\mathbb{P}(E)$ as follows.
\[
\mathbb{P}(E)=Proj(\oplus_{l\geq 0}{S^{l}E}).
\]
\end{definition}
\begin{definition}
Let $X$ be a scheme in characteristic $p$. We define the absolute Frobenius map of $X$ to be $F_X:X\rightarrow X$ such that $F_X=id_X$ as a map between two topological spaces and on each open set $U$, $F_{X}^{\sharp}:\mathcal{O}_X(U)\rightarrow \mathcal{O}_X(U)$ takes $f$ to $f^p$ for any $f\in \mathcal{O}_X(U)$.
\end{definition}
\begin{definition}
Let $S$ be a scheme in characteristic $p$ and $X$ be an $S$-scheme. Consider the following diagram:
 \[
 \xymatrix{
 	X \ar[drr]_{F_{X/S}}  \ar[rrrrd]^{F_X}  \ar[rrddd]_f\\
 	&&X^{(p)}\ar[rr] \ar[dd]^{f'}&&X\ar[dd]_f\\
 	\\
 	&&S\ar[rr]^{F_S}&&S,
 }
 \]
where $X^{(p)}$ is defined as the fibre product of $X$ and $S$ in the diagram. The induced map $F_{X/S}$ is called the Frobenius morphism of $X$ relative to $S$.	
\end{definition}
\begin{rem}\emph{(\cite{ref6} Lemma ~1.5)}\label{F}
\emph{Let $S$ be a noetherian scheme in characteristic $p$, $E$ be a locally free coherent sheaf on $S$ and $X=\mathbb{P}(E)$. Then $X^{(p^m)}=\mathbb{P}(F^{m^{\ast}}E)$.}
\end{rem}
One of the key tools in studying vector bundles in characteristic $p$ is the following lemma of Katz.
\begin{lemma}\emph{(\cite{ref11} Lemma~ 1.4)}\label{ka}
	Let $X$ and $Y$ be two varieties smooth over $S$, a noetherian scheme in characteristic $p$, and $f$ be a $S$-morphism from $X$ to $Y$. If the induced map on differentials, $df:f^{\ast}\Omega_{Y/S}\rightarrow\Omega_{X/S}$ is the zero map, then $f$ can be factored through the relative Frobenius morphism $F_{X/S}$.
\end{lemma}
\begin{definition}
Denote by $H_i$ the universal subbundle on the complete flag manifold $F(1,\cdots,n-1)$ of rank $i(1\leq i \leq n)$ and $X_i=c_1(H_i/H_{i-1})$.
\end{definition}
Although M.Guyot's paper \cite{ref15} is based on the field of characteristic zero, the following conclusions are also true in positive characteristic.
\begin{lemma}\emph{(\cite{ref15} Theorem ~3.2)}\label{zq}
Suppose $\mathbf{Z}[X_1,\ldots,X_{d-1};X_d;X_{d+1}]$ to be the ring of polynomials in $d+1$ variables with integral coefficients symmetrical in $X_1,\ldots,X_{d-1}$ and  $A(\bar{F})$ is the Chow ring of $\bar{F}$.
	The natural morphism $\mathbf{Z}[X_1,\ldots,X_{d-1};X_d;X_{d+1}]\rightarrow A(\bar{F})$ is surjective and its kernel is the ideal generated by
	$\sum_{i}(X_1,\ldots,X_{d+1}), (n-d-1)<i\leq n$, where\[\sum_{i}(X_1,\ldots,X_{d+1}):=\sum_{\alpha_1+\cdots+\alpha_{d+1}=i}X_1^{\alpha_1}\cdots X_{d+1}^{\alpha_{d+1}},\] and $\alpha_1,\ldots,\alpha_{d+1}$ are nonnegative integers.
 \end{lemma}
\begin{lemma}\emph{(\cite{ref15} Lemma~ 5.1)}\label{pic}
	The Picard group of $\bar{F}$ is generated by ${q_1}^{\ast}\mathcal{O}_{G}(1)$, $\mathcal{H}_{H_{d}^{\vee}}$ and $\mathcal{H}_{Q_{n-d}}$. The Chern polynomial \[c_{\mathcal{H}_{H_{d}^{\vee}}}(T)=T+X_d, ~c_{\mathcal{H}_{Q_{n-d}}}(T)=T-X_{d+1}.\]Here
	$c_{E}(T):=T^r-c_1(E)T^{r-1}+\cdots+(-1)^{r}c_r(E)$ is the chern polynomial of rank $r$-bundle $E$.
\end{lemma}
\begin{lemma}\emph{(\cite{ref15} Proposition ~2.3, 2.5)} \label{ta}
	The restriction of the relative cotangent bundle $\Omega_{\bar{F}/G}$ to every $q_2$-fibre $\widetilde{L}={q_2}^{-1}(L)\subseteq F$ has the following form\[
	\Omega_{\bar{F}/G}|\widetilde{L}=\mathcal{O}_{\widetilde{L}}(1)^{\oplus n-2}.
	\]
\end{lemma}
\section{Uniform vector bundles of rank r(r<d) on G }
In this section, we suppose $k$ is an algebraically closed field of arbitrary characteristic.
\begin{prop}\label{e}
Let $E$ be an algebraic vector bundle of rank $r$ over $G=G_k(d,n)$ and assume $E|L=\mathcal{O}_{L}^{\oplus r}$ for every line $L\subseteq G$. Then $E$ is trivial.	
\end{prop}
\begin{proof}
We prove the theorem by induction on $d$. For $d=1$, the Grassmannian is just $\mathbb{P}_k^{n-1}$, the result holds (see \cite{ref14} Theorem 3.2.1). Let's consider the diagram
\begin{align}
	\xymatrix{
	F(d-1,d)\ar[d]^{q_1}   \ar[r]^-{q_2} & G_k(d-1,n).\\
	G
}
\end{align}
It's not hard to see that every $q_2$-fibre $q_2^{-1}(x)$ is isomorphic to $\mathbb{P}_k^{n-d}$ and $q_1(q_2^{-1}(x))\cong \mathbb{P}_k^{n-d}$ is a $d$-dimensional linear subspace containing the $(d-1)$-dimensional linear subspace corresponding to $x$. By assumption, the restriction of $E$ to every line in $q_1(q_2^{-1}(x))$ is trivial, thus $E|q_1(q_2^{-1}(x))$ is trivial. Next, let's consider the coherent sheaf\[
E'={q_2}_{*}q_1^{*}E.
\]
Note that $E'$ is an algebraic vector bundle of rank $r$ over $G_k(d-1,n)$ and $q_1^{*}E\cong q_2^{*}E'$, because $q_1^{*}E|q_2^{-1}(x)\cong E|q_1(q_2^{-1}(x))$ is trivial on all $q_2$-fibres.

\textbf{Claim.} $E'|L=\mathcal{O}_{L}^{\oplus r}$ for every line $L$ in $G_k(d-1,n)$.

In fact, because $q_1^{*}E|q_1^{-1}(y)$ is trivial for every $y\in G$,\[
E'|q_2(q_1^{-1}(y))\cong q_2^{*}E'|q_1^{-1}(y)\cong q_1^{*}E|q_1^{-1}(y)
\]
is trivial for every $y\in G$. Since every line $L\subseteq G_k(d-1,n)$ is contained in some set $q_2(q_1^{-1}(y))$, the restriction $E'|L=\mathcal{O}_{L}^{\oplus r}$ for every line $L$ in $G_k(d-1,n)$.

By the induction hypothesis, $E'$ is trivial. Thus $q_1^{*}E\cong q_2^{*}E'$ is trivial, so is $E\cong {q_1}_{*}q_1^{*}E$.	
	\end{proof}

\begin{cor}\label{trivial}
If $E$ be a globally generated vector bundle of rank $r$ over $G$ with $c_1(E)=0$, then $E$ is trivial.
\end{cor}
\begin{proof}
Since $E$ is globally generated, we have an exact sequence
\[0\rightarrow K \rightarrow \mathcal{O}_G^{\oplus N}\rightarrow E\rightarrow 0.\]
Restricting this sequence to a line $L\subseteq G$ we get\[
0\rightarrow K|L \rightarrow \mathcal{O}_L^{\oplus N}\rightarrow E|L\rightarrow 0.
\]
Suppose $E|L=\oplus_{i=1}^r \mathcal{O}_L(a_i)$, then together with $E|L$ is globally generated, we have $a_i\ge 0, 1\le i\le r$. If $c_1(E)=0$, then we must have $a_i=0$ for all $i$. Thus $E$ is trivial on every line and hence trivial.
\end{proof}

\begin{thm}\label{b}
For $r<d$ every uniform $r$-bundle over $G$ splits as a sum of line bundles.
	\end{thm}
	\begin{proof}
We prove this theorem by induction on $r$. For $r=1$, there is nothing to prove. Suppose the assertion is true for all uniform $r'$-bundles with $1\leq r'<r<d$. If $E$ is a uniform $r$-bundle, after twisting with an appropriate line bundle and dualizing if necessary, we can assume that $E$ has the splitting type
\[\underline{a}_E=(a_1,\ldots,a_r),
\]	
where $a_1\leq\cdots\leq a_r$ and $a_1=\cdots=a_t=0, ~a_{t+1}>0$. If $t=r$, then $E$ is trivial by Proposition \ref{e}. Therefore let $t<r$, i.e.,
\[\underline{a}_E=(0,\ldots,0,a_{t+1},\ldots,a_r),~a_{t+i}>0, ~\text{for}~i=1,\ldots,r-t.
\]

Let's consider the standard diagram
\begin{align}\label{key}
\xymatrix{
	\bar{F}\ar[d]^{q_1}   \ar[rr]^-{q_2} && F(d-1,d+1).\\
	G
}
\end{align}
For $L\in F(d-1,d+1)$, the $q_2$-fibre
\[\widetilde{L}={q_2}^{-1}(L)=\{(x,L)|x\in L\},\]
is mapped isomorphically under $q_1$ to the line $L$ in $G$ and we have \[{q_1}^\ast E|\widetilde{L}\cong E|L.\]
For $x\in G$, the $q_1$-fibre over $x$,
\[{q_1}^{-1}(x)=\{(x,L)|x\in L\},
\]
is mapped isomorphically under $q_2$ to the subvariety
\[\text{VMRT}_x=\{L\in F(d-1,d+1)|x\in L\}\cong \mathbb{P}_k^{d-1}\times\mathbb{P}_k^{n-d-1}\]
(VMRT$_x$ means the variety of minimal rational tangents at $x$. We refer to \cite{ref22} for a complete account on the VMRT).
Because \[E|L\cong\mathcal{O}_L^{\oplus t}\oplus\bigoplus\limits_{i=1}^{r-t} \mathcal{O}_L(a_{t+i}),~ a_{t+i}>0,\]
\[h^0\left({q_2}^{-1}(L),{q_1}^\ast (E^{\vee})|_{q_2^{-1}(L)}\right)=t\] for all $L\in F(d-1,d+1)$.
Thus the direct image ${q_2}_{\ast}{q_1}^{\ast}(E^{\vee})$ is a vector bundle of rank $t$ over $F(d-1,d+1)$.
The canonical homomorphism of sheaves\[
{q_2}^{\ast}{q_2}_{\ast}P^{\ast}(E^{\vee})\rightarrow {q_1}^{\ast}(E^{\vee})
\]
makes $\widetilde{N}^{\vee}:={q_2}^{\ast}{q_2}_{\ast}{q_1}^{\ast}(E^{\vee})$ to be a subbundle of ${q_1}^{\ast}(E^{\vee})$. Because over each ${q_2}$-fibre $\widetilde{L}$, the evaluation map \[
\widetilde{N}^{\vee}|\widetilde{L}=H^0(\widetilde{L},{q_1}^{\ast}(E^{\vee})|\widetilde{L})\otimes_{k}\mathcal{O}_{\widetilde{L}}\rightarrow {q_1}^{\ast}(E^{\vee})|\widetilde{L}
\]
identifies $\widetilde{N}^{\vee}|\widetilde{L}$ with $\mathcal{O}_L^{\oplus t}\subseteq\mathcal{O}_L^{\oplus t}\oplus\bigoplus_{i=1}^{r-t} \mathcal{O}_L(-a_{t+i})=E^{\vee}|L.$
Over $\bar{F}$ we thus obtain an exact sequence\[
0\rightarrow\widetilde{M}\rightarrow {q_1}^{\ast}E\rightarrow\widetilde{N}\rightarrow0
\]
of vector bundles, whose restriction to ${q_2}$-fibres $\widetilde{L}$ looks as follows:
$$
\xymatrix{
0\ar[r] &\widetilde{M}|\widetilde{L} \ar[r] \ar[d]_\cong&	{q_1}^{\ast}E|\widetilde{L} \ar[r] \ar[d]_\cong&\widetilde{N}|\widetilde{L}\ar[r]\ar[d]_\cong&0\\
	0\ar[r] &\bigoplus_{i=1}^{r-t} \mathcal{O}_L(a_{t+i})\ar[r]  &\mathcal{O}_L^{\oplus t}\oplus\bigoplus_{i=1}^{r-t} \mathcal{O}_L(a_{t+i}) \ar[r] &\mathcal{O}_L^{\oplus t}\ar[r]&0.\\
	 }
$$
By Lemma \ref{KQ} below, there are bundles $M={q_1}_{\ast}\widetilde{M}$, $N={q_1}_{\ast}\widetilde{N}$ over $G$ with \[\widetilde{M}={q_1}^{\ast}M,\widetilde{N}={q_1}^{\ast}N.
\]
$M$ and $N$ are then necessarily uniform and we obtain by projecting the bundle sequence\[
0\rightarrow {q_1}^{\ast}M\rightarrow {q_1}^{\ast}E\rightarrow {q_1}^{\ast}N\rightarrow0
\]
onto $G$ to get the exact sequence
\begin{align}\label{c}
0\rightarrow M\rightarrow E\rightarrow N\rightarrow 0.
\end{align}
Because $M$ and $N$ are uniform vector bundles of rank smaller than $r$, by the induction hypothesis,
\[
M=\bigoplus\limits_{i=1}^{r-t} \mathcal{O}_{G}(a_{t+i}),N=\mathcal{O}_{G}^{\oplus t}.
\]
It follows from the Kempf vanishing theorem that $H^1(G,N^{\vee}\otimes M)=0$. Thus the exact sequence (\ref{c}) splits and hence also $E$.
	\end{proof}
\begin{lemma}\label{KQ}
	There are bundles $M$, $N$ over $G$ with $\widetilde{M}={q_1}^{\ast}M,\widetilde{N}={q_1}^{\ast}N$.
\end{lemma}
\begin{proof}
	To prove the lemma, it suffices to show that $\widetilde{M}$,$\widetilde{N}$ are trivial on all ${q_1}$-fibres (the canonical morphisms ${q_1}^{\ast}{q_1}_{\ast}\widetilde{M}\rightarrow\widetilde{M}$, ${q_1}^{\ast}{q_1}_{\ast}\widetilde{N}\rightarrow\widetilde{N}$ are then isomorphisms).
Because $\widetilde{N}^{\vee}$ is a subbundle of ${q_1}^{\ast}(E^{\vee})$ of rank $t$, for every point $x\in G$, it provides a morphism\[
\varphi:\text{VMRT}_x\rightarrow\mathbb{G}_k(t-1,\mathbb{P}_k(E^{\vee}_x)).
\]
We claim that $\psi:=\varphi|\mathbb{P}_k^{d-1}$ is constant for any $\mathbb{P}_k^{d-1}\subseteq \text{VMRT}_x$.

Let's consider $\psi^{\ast}H_t$ and $\psi^{\ast}Q_{r-t}$(the pull back of universal bundle $H_t$ and universal  quotient bundle $Q_{r-t}$ under $\psi$), which are vector bundles on $\mathbb{P}_k^{d-1}$. We have the exact sequence\[
0\rightarrow\psi^{\ast}H_t\rightarrow\mathcal{O}^{\oplus r}_{\mathbb{P}_k^{d-1}}\rightarrow\psi^{\ast}Q_{r-t}\rightarrow 0.
\]
Then \[
c(\psi^{\ast}H_t).c(\psi^{\ast}Q_{r-t})=1,
\]
i.e. $\big(1+c_1(\psi^{\ast}H_t)+\cdots+c_t(\psi^{\ast}H_t)\big).\big(1+c_1(\psi^{\ast}Q_{r-t})+\cdots+c_{r-t}(\psi^{\ast}Q_{r-t})\big)=1$.

Because $r<d$ and the Chow ring of $\mathbb{P}_k^{d-1}$ is $A(\mathbb{P}_k^{d-1})=\mathbb{Z}[\mathcal{H}]/{\mathcal{H}^d}$, where $\mathcal{H}$ is the rational equivalence class of a hyperplane, this must imply \[
c(\psi^{\ast}H_t)=1, ~  c(\psi^{\ast}Q_{r-t})=1.
\]
(Note that here we use the condition $r<d$, otherwise the proof doesn't work.)
In particular\[
c_1(\psi^{\ast}H_t)=0 , ~c_1(\psi^{\ast}Q_{r-t})=0.
\]
Let $\mathcal{O}_{\mathbb{G}_k(t-1,r-1)}(1)$ be the Pl$\ddot{u}$cker bundle of $\mathbb{G}_k(t-1,r-1)$, since deg $\psi^{\ast}\mathcal{O}_{\mathbb{G}_k(t-1,r-1)}(1)=c_1(\psi^{\ast}H_t)=0$, $\psi$ is constant.
Because $\text{VMRT}_{x}\cong \mathbb{P}_k^{d-1}\times\mathbb{P}_k^{n-d-1}$ can be covered by a family of $\mathbb{P}_k^{d-1}$ and for any two different $\mathbb{P}_k^{d-1}$'s in the family, there exists a $\mathbb{P}_k^{d-1}$ which intersects them both. So we obtain that $\varphi$ is constant. Thus $\widetilde{N}$ is trivial on all ${q_1}$-fibres. Moreover for every point $x\in G$, $\widetilde{M}^{\vee}|{q_1}^{-1}(x)$ is globally generated and $c_1(\widetilde{M}^{\vee}|{q_1}^{-1}(x))=0$, so $\widetilde{M}^{\vee}|{q_1}^{-1}(x)$ is trivial due to Corollay \ref{trivial}.
\end{proof}
\begin{cor}\label{d}
Let $E$ be a vector bundle on $G=G(d,n)$ $(d\ge 2)$ over an algebraically closed field. If $E$ splits as a direct sum of line bundles when it restricts to every $\mathbb{P}_k^{d}\subseteq G$, then $E$ splits as a direct sum of line bundles on $G$.
\end{cor}
\begin{proof}
The condition implies that $E$ is uniform. In fact, any line $L$ in $G$ is given by a $(d-1)$-dimensional vector space $V_{d-1}$ and a $(d+1)$-dimensional vector space $V_{d+1}$. Since $G(d, V_{d+1})$ and $\{W\in G(d, 2d)|W\supseteq V_{d-1}\}$ are two different subsets of $G$ which are both isomorphic to $\mathbb{P}^d$ and contain $L$, we can see that $E$ has uniform splitting type.

We use the same notations in Theorem \ref{b} to prove the corollary by induction on $r$ (the rank of $E$). If we have the exact sequence of vector bundles on $G$
\begin{align}\label{a}
0\rightarrow M\rightarrow E\rightarrow N\rightarrow 0,
\end{align}	
where the rank of $M$ and $N$ are smaller than $r$,
such that \[
M|Z=\bigoplus\limits_{i=1}^{r-t} \mathcal{O}_Z(a_{t+i}),N|Z=\mathcal{O}_Z^{\oplus t},
\]
 for every projective subspace $Z$ of dimension $d$, then by the induction hypothesis, $M$ and $N$ split. It follows from the Kempf vanishing theorem that $H^1(G,N^{\vee}\otimes M)=0$. Thus the above exact sequence splits and hence also $E$.

Similar to the proof of Theorem \ref{b}, we can obtain an exact sequence in $\bar{F}$\[
0\rightarrow\widetilde{M}\rightarrow {q_1}^{\ast}E\rightarrow\widetilde{N}\rightarrow0.	\]
If we prove that the morphism $\varphi$ is constant for every $x\in G$, then there exist two bundles $M$, $N$ over $G$ with $\widetilde{M}={q_1}^{\ast}M,\widetilde{N}={q_1}^{\ast}N$. By projecting the bundle sequence\[
0\rightarrow {q_1}^{\ast}M\rightarrow {q_1}^{\ast}E\rightarrow {q_1}^{\ast}N\rightarrow0
\]
onto $G$, we can get the desired exact sequence \ref{a}.
Thus, to prove the existence of the above exact sequence, it suffices to show that the map\[
 \varphi:\text{VMRT}_x\rightarrow\mathbb{G}_k(t-1,\mathbb{P}_k(E^{\vee}_x))
 \]
 is constant for every $x\in G$ .
 Given a projective subspace $Z$ of dimension $d$ and a line $L\subseteq Z$, we take any point $x\in L$ and denote by $Z'$ the subspace of $\text{VMRT}_x$ corresponding to the tangent directions to $Z$ at $x$. By the hypothesis, $E|Z$ is the direct sum of line bundles, so
 \[
  \varphi|Z':Z'\rightarrow\mathbb{G}_k(t-1,\mathbb{P}_k(E^{\vee}_x))
 \]
 is constant. Since $G$ covered by linear projective subspaces of dimension $d$ and $\text{VMRT}_{x}\cong \mathbb{P}_k^{d-1}\times\mathbb{P}_k^{n-d-1}$, $\varphi$ is constant for every $x\in G$ by the same reason given in the proof of Lemma \ref{KQ}.
\end{proof}
\section{Uniform vector bundles of rank d on G }
In this section, we suppose the characteristic of the field $k$ is positive integer $p$. Although Guyot's paper \cite{ref15} is based on the field of characteristic zero, many conclusions are also true for positive characteristic, except for the argument of the splitting type of uniform non-splitting $d$-bundle over $G$ being $(0,\ldots,0,-1)$. The main reason is that the first part of Guyot's paper, which is independent of the characteristic of the field $k$, is for studying the Chow ring and axiomatically computing Chern classes of vector bundles on flag varieties.
		
The uniform vector bundle $E$ can be characterized as follows\cite{ref15}. Let (\ref{key}) be the stardard diagram and $L$ be a line in $G$, then $E$ is uniform with
$$E|L=\mathcal{O}_L(u_1)^{\oplus r_1}\oplus\cdots\oplus\mathcal{O}_L(u_t)^{\oplus r_t}, ~u_1>\cdots>u_t,$$
if and only if there is a filtration
$$0=H N_{q}^{0}\left({q_1}^{*} E\right)\subseteq H N_{q}^{1}\left({q_1}^{*} E\right)\subseteq\cdots\subseteq H N_{q}^{t}\left({q_1}^{*} E\right)={q_1}^*E$$
of ${q_1}^{\ast}E$ by subbundles $ H N_{q}^{i}\left({q_1}^{*} E\right)$ such that $H N_{q}^{i}\left({q_1}^{*} E\right)/H N_{q}^{i-1}\left({q_1}^{*} E\right)\cong {q_2}^*(E_i)\otimes O_{q}(u_i)$, where
$E_i$ is an algebraic vector bundle of rank $r_i$ over $F(d-1, d+1)$, $O_{q}(1)= \mathcal{H}_{H_{d}^{\vee}}$ and\[
H N_{q}^{i}\left({q_1}^{*} E\right)=\operatorname{Im}\left[ {q_2}^{*}  {q_2}_{*}\left({q_1}^{*} E \otimes O_{q}\left(-u_{i}\right)\right) \otimes O_{q}\left(u_{i}\right) \rightarrow {q_1}^{*} E\right].
\]

This filtration is the relative Harder-Narasimhan filtration \nocite{ref12,ref13,ref14}  of ${q_1}^*E$.
Since rank $E=d\le n-d$, by Whitney's formula and Lemma \ref{zq}, we have
\begin{equation}\label{chow}
c_{{q_1}^{*} E}(T)=\prod_{i=1}^{t}c_{H N_{q}^{i}\left({q_1}^{*} E\right)/H N_{q}^{i-1}\left({q_1}^{*} E\right)}(T)+a\sum_{n-d}(X_1,\ldots,X_{d+1}).
\end{equation}
\begin{thm}\label{z}
If $E$ is a uniform vector bundle over $G$ of rank $d$, then $E\cong\oplus_{i=1}^{d}\mathcal{O}_G(a_i)$, $E\cong {F^m}^{\ast}H_d\otimes\mathcal{O}_G(v_1)$ or $E\cong {F^m}^{\ast}(H_d^{\vee})\otimes\mathcal{O}_G(v_2)$, $m\geq0$ and $a_i, v_1, v_2 \in \mathbb{Z}$.
\end{thm}
\begin{proof}
If $a=0$ in the equality (\ref{chow}) and $E$ can't split as a sum of line bundles, then by the assertion of Guyot (\cite{ref15} Corollary 4.1.1 b), the expression of $c_{{q_1}^{*} E}(T)$ can only contain $u_1$ and $u_2$ and $r_1=d-1$, $r_2=1$. So, after twisting with an appropriate power of $\mathcal{O}_{G}(1)$ and dualizing if necessary, we can let $u_1=0$ and assume $E$ is of type $(0,0,\ldots,0,b), b<0$. Then we can write \[
c_{{q_1}^{*} E}(T)=\prod_{i=1}^{d}(T+bX_i).
\]
So we get \[
c_{H N_{q}^{1}\left({q_1}^{*} E\right)}(T)=\prod_{i=1}^{d-1}(T+bX_i)~ \text{and} ~  c_{{q_1}^{*} E/H N_{q}^{1}\left({q_1}^{*} E\right)}(T)=T+bX_d.
\]
By lemma \ref{pic}, we get ${q_1}^{*} E/H N_{q}^{1}\left({q_1}^{*} E\right)\cong \big(({\mathcal{H}_{H_{d}^{\vee}})^{\vee}}\big)^{\otimes(-b)}$. Hence on $\bar{F}$, we have the following exact sequence
\begin{equation}\label{o}
0\rightarrow H N_{q}^{1}\left({q_1}^{*} E\right)\rightarrow {q_1}^\ast E\rightarrow  \big(({\mathcal{H}_{H_{d}^{\vee}})^{\vee}}\big)^{\otimes(-b)}\rightarrow 0.
\end{equation}
By the universal property of $\mathbb{P}(E)$, there is a unique $G$-morphism $\sigma:\bar{F}\rightarrow \mathbb{P}(E)$ such that \[
\sigma^{\ast}\mathcal{O}_{\mathbb{P}(E)}(1)= \big((\mathcal{H}_{H_{d}^{\vee}})^{\vee}\big)^{\otimes(-b)}, \sigma^{\ast}\Omega_{\mathbb{P}(E)/G}=H N_{q}^{1}\left({q_1}^{*} E\right)\otimes \big(\mathcal{H}_{H_{d}^{\vee}}\big)^{\otimes (-b)}.
\]

Let's consider the following diagram

\begin{align}
	\xymatrix{
		F_1   \ar[rrdd]_{q_{11}} && \bar{F}\ar[dd]^{q_1}\ar[ll]_{pr_1} \ar[rr]^{pr_2}  && F_2 \ar[lldd]^{q_{12}} \\
		\\
		&&G.
	}
\end{align}

Case 1. $b=-1$:

Projecting the exact sequence\[
0\rightarrow H N_{q}^{1}\left({q_1}^{*} E\right)\rightarrow {q_1}^\ast E\rightarrow  ({\mathcal{H}_{H_{d}^{\vee}})^{\vee}}\rightarrow 0
\]
onto $F_1$ and by Remark \ref{g}, we get the exact sequence ($R^1{pr_1}_{*}H N_{q}^{1}\left({q_1}^{*} E\right)=0$)
\[
0\rightarrow {pr_1}_{*}H N_{q}^{1}\left({q_1}^{*} E\right)\rightarrow {q_{11}}^\ast E\rightarrow \mathcal{O}_{F_1}(1)\rightarrow 0.
\]
Restricting the above exact sequence to a fibre of ${q_{11}}$, ${q_{11}}^{-1}(x)\cong \mathbb{P}_k^{d-1}$, we see that \[
 {pr_1}_{*}H N_{q}^{1}\left({q_1}^{*} E\right)|_ {q_{11}^{-1}(x)}\cong \Omega_{\mathbb{P}_k^{d-1}}(1).
\]
Hence we get ${q_{11}}_{*}({pr_1}_{*}H N_{q}^{1}\left({q_1}^{*} E\right))=\textbf{R}^{1}{q_{11}}_{*}({pr_1}_{*}H N_{q}^{1}\left({q_1}^{*} E\right))=0$. In particular, \[E\cong {q_{11}}_{*}\mathcal{O}_{F_1}(1)\cong H_d.\]
Case 2. $b<-1$:

By restricting the induced map $d\sigma: \sigma^{\ast}\Omega_{\mathbb{P}(E)/G}=H N_{q}^{1}\left({q_1}^{*} E\right)\otimes \big({\mathcal{H}_{H_{d}^{\vee}}}\big)^{\otimes(-b)}\rightarrow\Omega_{\bar{F}/G}$ to any $q_2$-fibre $\widetilde{L}={q_2}^{-1}(L)\subseteq \bar{F}$, and by Lemma \ref{ta}, we get \[
d\sigma|\widetilde{L}:\mathcal{O}_{\widetilde{L}}(-b)^{\oplus d-1}\rightarrow\mathcal{O}_{\widetilde{L}}(1)^{\oplus n-2}.
\]
Because $b<-1$, we have $d\sigma|{q_2}^{-1}(L)=0$ for all $L\in F(d-1,d+1)$, i.e. $d\sigma=0$.  By Lemma \ref{ka}, $\sigma$ can be factored through the relative Frobenius morphism $F_{\bar{F}/G}^m$ for some positive integer $m$:
\[
\xymatrix{
	\bar{F} \ar@/^2pc/[rrrr]^{\sigma}   \ar[rrdd]_{q_1}\ar[rr]^{F_{\bar{F}/G}^m} && \bar{F}^{(p^m)}\ar[dd]^{q_1'} \ar[rr]^{\sigma'} && \mathbb{P}(E) \ar[lldd]^{\pi} \\
	\\
	&&G.
}
\]
Let's consider the following diagram:
\begin{align}
\xymatrix{
	\bar{F}_1^{(p^m)}   \ar[rrdd]_{q_{11}'}  && \bar{F}^{(p^m)}\ar[dd]^{q_1'}\ar[ll]_{pr_1'}\ar[rr]^{pr_2'} && \bar{F}_2^{(p^m)}  \ar[lldd]^{q_{12}'}. \\
	\\
	&&G.
}
\end{align}
By Remark \ref{F}, we have $\bar{F}_1^{(p^m)}=\mathbb{P}(\bar{F}^{m^{\ast}}(H_d))$, $\bar{F}_2^{(p^m)}=\mathbb{P}(\bar{F}^{m^{\ast}}(Q_{n-d}^{\vee}))$.

Because  $\sigma$ can be factored through the relative Frobenius morphism $F_{\bar{F}/G}^m$, hence on $\bar{F}^{(p^m)}$, we have the exact sequence  \nocite{guyot1985caracterisation}
\begin{equation}\label{q}
0\rightarrow (H N_{q}^{1}\left({q_1}^{*} E\right))'\rightarrow {q_1'}^\ast E\rightarrow\big(({\mathcal{H}'_{H_{d}^{\vee}})^{\vee}}\big)^{\otimes{(\frac{-b}{p^m}})}\rightarrow0.
\end{equation}
where $\mathcal{H}'_{H_{d}^{\vee}}$ is the tautological bundle on $\bar{F}^{(p^m)}$ associated to $\bar{F}_1^{(p^m)}$, and the pull back of the exact sequence (\ref{q}) under $F_{\bar{F}/G}^m$ is the exact sequence (\ref{o}).
We also have the reduced map
 \[d{\sigma'}:{\sigma'}^\ast\Omega_{\mathbb{P}(E)/G}\rightarrow \Omega_{\bar{F}^{(p^m)}/G}.\]
By restricting the map to any  $q_2$-fiber, we get
\[
d\sigma'|\widetilde{L}:\mathcal{O}_{\widetilde{L}}(-b)^{\oplus d-1}\rightarrow\mathcal{O}_{\widetilde{L}}(p^m)^{\oplus n-2}.
\]
By Lemma \ref{ka}, we may assume $-b\leq p^m$. On the other hand, we have $p^m|-b$, thus $-b=p^m$. On $\bar{F}^{(p^m)}$, we now have the following exact sequence\[
0\rightarrow (H N_{q}^{1}\left({q_1}^{*} E\right))'\rightarrow {q_1'}^\ast E\rightarrow({\mathcal{H}'_{H_{d}^{\vee}})^{\vee}}\rightarrow0.
\]
By projecting the exact sequence to $\bar{F}_1^{(p^m)}$, we get an exact sequence ($R^1{pr_1'}_{*}(H N_{q}^{1}\left({q_1}^{*} E\right))'=0$)
\[
0\rightarrow {pr_1'}_{*}(H N_{q}^{1}\left({q_1}^{*} E\right))'\rightarrow {q_{11}'}^\ast E\rightarrow \mathcal{O}_{\bar{F}_1^{(p^m)}}(1)\rightarrow 0.
\]
As in Case 1, we find ${q_{11}'}_{*}({pr_1'}_{*}(H N_{q}^{1}\left({q_1}^{*} E\right))')=\textbf{R}^{1}{q_{11}'}_{*}({pr_1'}_{*}(H N_{q}^{1}\left({q_1}^{*} E\right))')=0$. In particular, \[
E\cong {q_{11}'}_{\ast}\mathcal{O}_{\bar{F}_1^{(p^m)}}(1)\cong {F^m}^{\ast}H_d.\]
If $a\neq0$, the equation (\ref{chow}) implies that $d=n-d$. Suppose that $E$ can't split as a sum of line bundles. After twisting with an appropriate power of $\mathcal{O}_{G}(1)$ and dualizing if necessary, we can assume $E$ is of type $(0,0,\ldots,0,\beta)(\beta<0)$. By Proposition \ref{prop} below, we get $E\cong {F^m}^{\ast}Q_{n-d}^{\vee}(m\geq 0)$, where $Q_{n-d}^{\vee}$ can be viewed as the universal subbundle of the dual Grassmannian. Hence we get the desired result.
 \end{proof}

\begin{prop}\label{prop}
Let $E$  be a uniform vector bundles over $G$ of rank $n-d$, and $a\neq0$. Then $E\cong\oplus_{i=1}^{n-d}\mathcal{O}_G(a_i)$, $E\cong {F^m}^{\ast}Q_{n-d}^{\vee}\otimes\mathcal{O}_G(v_1)$ or $E\cong {F^m}^{\ast}Q_{n-d}\otimes\mathcal{O}_G(v_2)$, where $m\geq 0$ and $v_1, v_2\in \mathbb{Z}$.		
\end{prop}
\begin{proof}
By the assertion of Guyot (\cite{ref15} Lemma 4.2.3), if $E$ can't split as a sum of line bundles, then there are two cases:

1) $r_1=n-d-1$, $r_2=1$: after tensoring $E$ with a line bundle, we may assume $c_{{q_1}^{*}E}(T)=\sum_{n-d}(T,\beta X_1,\ldots,\beta X_d)~(\beta<0).$

2) $r_1=1$, $r_2=n-d-1$: after tensoring $E$ with a line bundle, we may assume $c_{{q_1}^{*}E}(T)=\sum_{n-d}(T,-\beta X_1,\ldots,-\beta X_d)~(\beta<0).$

In the first case, we can write
\begin{eqnarray}
&&c_{{q_1}^{*}E}(T)-\beta^{n-d}\sum_{n-d}( X_1,\ldots, X_{d+1})\\
&=&(T-\beta X_{d+1})\big(\sum_{n-d-1}(T,\beta X_1,\ldots,\beta X_{d+1})\big).
\end{eqnarray}
So we get \[
c_{H N_{q}^{1}\left({q_1}^{*} E\right)}(T)=\sum_{n-d-1}(T,\beta X_1,\ldots,\beta X_{d+1}) ~\text{and}~  c_{{q_1}^{*} E/H N_{q}^{1}\left({q_1}^{*} E\right)}(T)=T-\beta X_{d+1}.
\]
By Lemma \ref{pic}, we get ${q_1}^{*} E/H N_{q}^{1}\left({q_1}^{*} E\right)\cong \big(({\mathcal{H}_{Q_{n-d}})^{\vee}}\big)^{\otimes(-\beta)}$. Hence on $\bar{F}$, we have the following exact sequence
\begin{equation}
0\rightarrow H N_{q}^{1}\left({q_1}^{*} E\right)\rightarrow {q_1}^\ast E\rightarrow  \big(({\mathcal{H}_{Q_{n-d}})^{\vee}}\big)^{\otimes(-\beta)}\rightarrow 0.
\end{equation}
According to the proof in the above theorem, we get $E\cong {F^m}^{\ast}Q_{n-d}^{\vee}$.
For the second case, we get $E\cong {F^m}^{\ast}Q_{n-d}~(m\geq 0)$ for the similar reason.
\end{proof}

 Therefore, we have proved Theorem \ref{m} completely.

According to \cite{ref24} Section 3.3, the family of semistable torsion free sheaves over normal projective spaces in characteristic zero forms a bounded family. But it is still unknown whether the same result is true in characteristic $p$. But we conjecture it is true over Grassmannian $G$.

\begin{conj}
Denote $\mathfrak{F}^{ss}_G(r, c_1, c_2)$ to be the class of all semistable (in any sense) torsion free sheaves $\mathcal{F}$ on $G$ of rank $r$ with $c_i(\mathcal{F})=c_i$ for $i=1,2$ in characteristic $p$. Then $\mathfrak{F}^{ss}_G(r, c_1, c_2)$ forms a bounded family.
\end{conj}

\begin{rem}
\emph{Langer (\cite{ref25}) proved that the family of strongly semistable torsion free sheaves over normal projective spaces in arbitrary characteristic form a bounded family. But it is still unknown that, even for Grassmannians, if the statement is true for general semistable torsion free sheaves in characteristic $p$.}
\end{rem}

\section{Vector bundles on flag varieties}\label{flag}
Let $F:=F(d_1,\ldots,d_s)$ be the flag manifold parameterizing flags \[V_{d_1}\subseteq\cdots\subseteq V_{d_s}\subseteq k^n\]
where $dim(V_{d_i})=d_i, 1\le i \le s$. In this section, we suppose that the characteristic of $k$ is zero.

For any integer $i (1\le i \le s)$, there is a flag $V_{d_1}\subseteq\cdots V_{d_{i-1}}\subseteq V_{d_i-1}\subseteq V_{d_i+1}\subseteq V_{d_{i+1}}\subseteq\cdots V_{d_s}\subseteq k^n$. It's not hard to see that the set of flags $V_{d_1}\subseteq\cdots V_{d_{i-1}}\subseteq W\subseteq V_{d_{i+1}}\subseteq\cdots V_{d_s}\subseteq k^n$ such that\[
V_{d_i-1}\subseteq W\subseteq V_{d_i+1},~dim(W)=d_i
\]
is the projectivization of the quotient space $V_{d_i+1}/V_{d_i-1}$, so the set of such flags is isomorphic to $\mathbb{P}_k^1$. It follows that any $l\in F(d_1,\ldots,d_{i-1},d_i-1,d_i+1,d_{i+1},\ldots,d_s)$ determines a line $L\subseteq F$. We denote by $F^{(i)}:=F(d_1,\ldots,d_{i-1},d_i-1,d_i+1,d_{i+1},\ldots,d_s)$ the $i$-th irreducible component of the manifold of lines in $F$. (We re-emphasize that if the two adjacent integers in the expression of flag varieties such as $F^{(i)}$ are equal, we keep only one of them.)

We can consider $F^{(i)}$ in the following two cases:

Case I: $d_i-1=d_{i-1}$ and $d_i+1=d_{i+1}$, then we have the natural projection $F\rightarrow F^{(i)}$.

Case II: $d_i-1\neq d_{i-1}$ or $d_i+1\neq d_{i+1}$, then we have the \emph{standard diagram}
\begin{align}
\xymatrix{
	F(d_1,\ldots,d_{i-1},d_i-1,d_i,d_i+1,d_{i+1},\ldots d_s)\ar[d]^{q_1}   \ar[r]^-{q_2} & F^{(i)}\\
	F=F(d_1,\ldots,d_s).
}
\end{align}
On the flag manifold $F$, we denote by $H_{d_i}$ the \emph{universal subbundle} whose fiber at a point $[\Lambda]=[V_{d_1}\subseteq\cdots\subseteq V_{d_s}\subseteq k^n]\in F$ is the subspace $V_{d_i}$; that is,
		\[(H_{d_i})_{[\Lambda]}=V_{d_i}.\]
Set $H_{d_i,d_j}=H_{d_j}/H_{d_i}(j\ge i)$  as the \emph{universal quotient bundle}.

Let $E$ be an algebraic $r$-bundle over $F$. According to the theorem of Grothendieck, for every $l\in F^{(i)}$, there is an $r$-tuple\[
a_E^{(i)}(l)=(a_1^{(i)}(l),\ldots,a_r^{(i)}(l))\in\mathbb{Z}^{r}~\text{with}~ a_1^{(i)}(l)\geq\cdots\geq a_r^{(i)}(l)
\]
such that $E|L\cong\bigoplus_{j=1}^{r}\mathcal{O}_{L}(a_j^{(i)}(l))$. We give $\mathbb{Z}^{r}$ the lexicographical ordering, i.e., $(a_1,\ldots,a_r)\le(b_1,\ldots,b_r)$ if the first non-zero difference $b_i-a_i$ is positive. Let \[
\underline{a}_E^{(i)}=\inf_{l\in F^{(i)}}a_E^{(i)}(l).
\]
\begin{definition}\label{UEi}
$\underline{a}_E^{(i)}$ is the generic splitting type of $E$ with respect to $F^{(i)}$, $S_E^{(i)}=\{l\in F^{(i)}|a_E^{(i)}(l)>\underline{a}_E^{(i)}\}$ is the set of jump lines with respect to $F^{(i)}$. We define $U_E^{(i)}:=F^{(i)}\backslash S_E^{(i)}$.
\end{definition}
\begin{rem}\label{u}
\emph{Fix integer $i, 1\le i \le s$. Let \[
M_t(a_1,\ldots,a_t)=\{l\in F^{(i)}|(a_1^{(i)}(l),\ldots,a_t^{(i)}(l))>(a_1,\ldots,a_t)\}.	
		\]
Because of the semicontinuity theorem, the set $M_1(a_1)=\{l\in F^{(i)}|h^0(L,E(-a_1-1)|L>0)\}$ is Zariski-closed in $F^{(i)}$. By induction on $t$ we see that $S_E^{(i)}=M_r{(\underline{a}_E^{(i)})}$ is Zariski-closed in $F^{(i)}$. Thus $U_E^{(i)}$ is a non-empty Zariski-open subset of $F^{(i)}$ (see \cite{ref14} Lemma 3.2.2).}
\end{rem}
\begin{definition}(\cite{ref21} p.128 or ~\cite{ref23} Proposition 1.6)\label{normal}
Let $X$ be a nonsingular variety over $k$. A coherent sheaf $\mathcal{F}$ over $X$ is called normal if for every open $U\subseteq X$ and every closed subset $A\subseteq U$ of codimension at least $2$, the restriction map\[
\mathcal{F}(U)\rightarrow \mathcal{F}(U\backslash A)
\]
is an isomorphism.	
\end{definition}	
The following result holds over any algebraically closed field.
\begin{prop}\label{y}
Let $E$ be an algebraic vector bundle of rank $r$ over $F:=F(d_1,\ldots,d_s)$ and assume $E|L=\mathcal{O}_{L}^{\oplus r}$ for every line $L$. Then $E$ is trivial.	
\end{prop}
\begin{proof}
	We prove the theorem by induction on $s$. For $s=1$, the flag manifold is just $G(d_1,n)$, the result holds by Proposition \ref{e}. Suppose the assertion is true for all flag manifolds $F(d_1',\cdots,d_{s-1}')$. Let's consider the natural projection \[
q:F=F(d_1,\ldots,d_s)\rightarrow F(d_2,\ldots,d_s).\]
	It's not hard to see that every $q$-fibre $q^{-1}(x)$ is isomorphic to the Grassmannian $G(d_1,d_2)$. Since the restriction of $E$ to every line in $q$-fibre $q^{-1}(x)$ is trivial by assumption, $E$ is trivial on all $q$-fibres by Proposition \ref{e}. It follows that $E'=q_{*}E$ is an algebraic vector bundle of rank $r$ over $F(d_2,\ldots,d_s)$ and $E\cong q^{*}E'$.
	
\textbf{Claim.} $E'|L=\mathcal{O}_{L}^{\oplus r}$ for every line $L$ in $F(d_2,\ldots,d_s)$.

In fact, let $L$ be a line in the $i$-th $(2\le i\le s)$ irreducible component of the manifold of lines in $F$. Then $q(L)$ is a line in the ${i-1}$-th $(2\le i\le s)$ irreducible component of the manifold of lines in $F(d_2,\ldots,d_s)$. When $L$ runs through all lines in the $i$-th $(2\le i\le s)$ component of the manifold of lines in $F$, $q(L)$ also runs through all lines in the $(i-1)$-th $(2\le i\le s)$ component of the manifold of lines in $F(d_2,\ldots,d_s)$. The projection $q$ induces an isomorphism\[
E'|q(L)\cong q^{*}E'|L\cong E|L.
\]
Since $E|L$ is trivial for all lines $L$ in the $i$-th $(2\le i\le s)$ irreducible component of the manifold of lines in $F$ by assumption, $E'|L$ is trivial for all lines $L$ in the $i$-th $(1\le i\le s-1)$ irreducible component of the manifold of lines in $F(d_2,\ldots,d_s)$. It follows that $E'|L=\mathcal{O}_{L}^{\oplus r}$ for every line $L$ in $F(d_2,\ldots,d_s)$.

By the induction hypothesis, $E'$ is trivial. Thus $E\cong q^{*}E'$ is trivial.
\end{proof}
\begin{lemma}(Descent Lemma \cite{ref19})\label{Descent}
Let $X$, $Y$ be nonsingular varieties over $k$, $f:X\rightarrow Y$ be a surjective submersion with connected fibres and $E$ be an algebraic $r$-bundle over $Y$. Let $\widetilde{K}\subseteq f^{\ast}E$ be a subbundle of rank $t$ in $f^{\ast}E$ and $\widetilde{Q}=f^{\ast}E/\widetilde{K}$ be its quotient. If \[
Hom(T_{X/Y},\mathcal{H}om(\widetilde{K},\widetilde{Q}))=0,
\]
then $\widetilde{K}$ is the form $\widetilde{K}=f^{\ast}K$ for some algebraic subbundle $K\subseteq E$ of rank $t$.
\end{lemma}
\begin{lemma}\label{rt}
Let $\widetilde{F^{(i)}}:=F(d_1,\ldots,d_{i-1},d_i-1,d_i,d_i+1,d_{i+1},\ldots,d_s)$,
 $\widetilde{L}={q_2}^{-1}(l)\subseteq\widetilde{F^{(i)}}$ for $l\in F^{(i)}$. If $F^{(i)}$ is in Case II, then for the relative cotangent bundle $\Omega_{\widetilde{F^{(i)}}/F}$, we have\[
\Omega_{\widetilde{F^{(i)}}/F}|\widetilde{L}=\mathcal{O}_{\widetilde{L}}(1)^{\oplus d_{i+1}-d_{i-1}-2}.
\]
\end{lemma}
\begin{proof}
By the definition of vector bundle $H_{d_i,d_j}(j>i)$, it's easy to check that on $\widetilde{F^{(i)}}$, we have the following two exact sequences:
\begin{align}\label{1}
0\rightarrow H_{d_i-1,d_i}^{\vee}\rightarrow {q_1}^{\ast}H_{d_{i-1},d_i}^{\vee} \rightarrow {q_2}^{\ast}H_{d_{i-1},d_i-1}^{\vee} \rightarrow 0,
\end{align}
\begin{align}
0\rightarrow H_{d_i,d_i+1}\rightarrow {q_1}^{\ast}H_{d_i,d_{i+1}}  \rightarrow {q_2}^{\ast}H_{d_i+1,d_{i+1}} \rightarrow 0.
\end{align}

Let $\widetilde{F_1^{(i)}}:=F(d_1,\ldots,d_{i-1},d_i-1,d_i,d_{i+1},\ldots d_s)$, $\widetilde{F_2^{(i)}}:=F(d_1,\ldots,d_{i-1},d_i,d_i+1,d_{i+1},\ldots d_s)$ and consider the following diagram
\begin{align}
\xymatrix{
\widetilde{F_1^{(i)}}  \ar[rrdd]_{q_{11}} && \widetilde{F^{(i)}}\ar[ll]_{pr_1}\ar[rr]^{pr_2}\ar[dd]^{q_1}  && \widetilde{F_2^{(i)}}  \ar[lldd]^{q_{12}} \\
	\\
	&&F.
}
\end{align}
All morphisms in the above diagram are projections. It is not hard to see that $\widetilde{F_1^{(i)}}=\mathbb{P}(H_{d_{i-1},d_i})$, $\widetilde{F_2^{(i)}}=\mathbb{P}(H_{d_i,d_{i+1}}^{\vee})$ and $H_{d_i-1,d_i}^{\vee}~(resp.~  H_{d_i,d_i+1})$ is the tautological line bundle on $\widetilde{F^{(i)}}$ associated to $\widetilde{F_1^{(i)}}~(resp.~ \widetilde{F_2^{(i)}})$, i.e. \[{pr_1}_{*}H_{d_i-1,d_i}^{\vee}=\mathcal{O}_{\widetilde{F_1^{(i)}}}(-1)~(resp.~ {pr_2}_{*}H_{d_i,d_i+1}=\mathcal{O}_{\widetilde{F_2^{(i)}}}(-1)).\]
Projecting the exact sequence \ref{1} onto $\widetilde{F_1^{(i)}}$ and considering the relative Euler sequence, we have the diagram of exact sequences (${R^1pr_1}_{*}H_{d_i-1,d_i}^{\vee}=0$)
$$
\xymatrix{
	0\ar[r] &{pr_1}_{*}H_{d_i-1,d_i}^{\vee}\ar[r] \ar[d]_\cong&	q_{11}^{\ast}H_{d_{i-1},d_i}^{\vee} \ar[r] \ar[d]_\cong&{pr_1}_{*}{q_2}^{\ast}H_{d_{i-1},d_i-1}^{\vee}\ar[r]\ar[d]_\cong&0\\
	0\ar[r] &\mathcal{O}_{\widetilde{F_1^{(i)}}}(-1) \ar[r]  &q_{11}^{\ast}H_{d_{i-1},d_i}^{\vee}\ar[r] & \mathcal{O}_{\widetilde{F_1^{(i)}}}(-1)\otimes T_{\widetilde{F_1^{(i)}}/F}\ar[r]&0.
}
$$
We get $T_{\widetilde{F_1^{(i)}}/F}\cong {pr_1}_{*}(H_{d_i-1,d_i}\otimes {q_2}^{\ast}H_{d_{i-1},d_i-1}^{\vee})$, since $H_{d_i-1,d_i}={pr_1}^{*}\mathcal{O}_{\widetilde{F_1^{(i)}}}(1)$.

Similarly, on $\widetilde{F_2^{(i)}}$, we have the exact sequences ($R^1{pr_2}_{*}H_{d_i,d_i+1}=0$)
$$
\xymatrix{
	0\ar[r] &{pr_2}_{*}H_{d_i,d_i+1}\ar[r] \ar[d]_\cong&	q_{12}^{\ast}H_{d_i,d_{i+1}} \ar[r] \ar[d]_\cong&{pr_2}_{*}{q_2}^{\ast}H_{d_i+1,d_{i+1}}\ar[r]\ar[d]_\cong&0\\
	0\ar[r] &\mathcal{O}_{\widetilde{F_2^{(i)}}}(-1) \ar[r]  &q_{12}^{\ast}H_{d_i,d_{i+1}}\ar[r] & \mathcal{O}_{\widetilde{F_2^{(i)}}}(-1)\otimes T_{\widetilde{F_2^{(i)}}/F}\ar[r]&0.
}
$$
We get $T_{\widetilde{F_2^{(i)}}/F}\cong {pr_2}_{*}(H_{d_i,d_i+1}^{\vee}\otimes {q_2}^{\ast}H_{d_i+1,d_{i+1}})$.
Since $\widetilde{F^{(i)}}$ as $F$-scheme is the fiber product of two $F$-schemes $\widetilde{F_1^{(i)}}$ and $\widetilde{F_2^{(i)}}$, we have
\begin{align}
T_{\widetilde{F^{(i)}}/F}
&\cong \big({pr_1}^{*}T_{\widetilde{F_1^{(i)}}/F}\big)\oplus\big({pr_2}^{*} T_{\widetilde{F_2^{(i)}}/F}\big)\\
&\cong \big({pr_1}^{*}{pr_1}_{*}(H_{d_i-1,d_i}\otimes {q_2}^{\ast}H_{d_{i-1},d_i-1}^{\vee})\big)\oplus\big({pr_2}^{*}{pr_2}_{*} (H_{d_i,d_i+1}^{\vee}\otimes {q_2}^{\ast}H_{d_i+1,d_{i+1}})\big).
\end{align}

The canonical homomorphism ${pr_1}^{*}{pr_1}_{*}(H_{d_i-1,d_i}\otimes {q_2}^{\ast}H_{d_{i-1},d_i-1}^{\vee})\cong H_{d_i-1,d_i}\otimes {q_2}^{\ast}H_{d_{i-1},d_i-1}^{\vee}$, because over each $pr_1$-fibre ${pr_1}^{-1}(l)$, the evaluation map
\begin{align}
&{pr_1}^{*}{pr_1}_{*}(H_{d_i-1,d_i}\otimes {q_2}^{\ast}H_{d_{i-1},d_i-1}^{\vee})|{pr_1}^{-1}(l)\\
&\cong H^0\big({pr_1}^{-1}(l),H_{d_i-1,d_i}\otimes {q_2}^{\ast}H_{d_{i-1},d_i-1}^{\vee}|{pr_1}^{-1}(l)\big)\otimes_{k}\mathcal{O}_{{pr_1}^{-1}(l)}\\
&\cong (H_{d_i-1,d_i}\otimes {q_2}^{\ast}H_{d_{i-1},d_i-1}^{\vee})|{pr_1}^{-1}(l).
\end{align}
is an isomorphism. Similarly, ${pr_2}^{*}{pr_2}_{*} (H_{d_i,d_i+1}^{\vee}\otimes {q_2}^{\ast}H_{d_i+1,d_{i+1}})\cong H_{d_i,d_i+1}^{\vee}\otimes {q_2}^{\ast}H_{d_i+1,d_{i+1}}$.
Hence,\[
\Omega_{\widetilde{F^{(i)}}/F}\cong (H_{d_i-1,d_i}^{\vee}\otimes {q_2}^{\ast}H_{d_{i-1},d_i-1})\oplus (H_{d_i,d_i+1}\otimes {q_2}^{\ast}H_{d_i+1,d_{i+1}}^{\vee})
\]
Finally, we get \[
\Omega_{\widetilde{F^{(i)}}/F}|\widetilde{L}=\mathcal{O}_{\widetilde{L}}(1)^{\oplus d_{i+1}-d_{i-1}-2}.
\]
\end{proof}
\begin{thm}\label{main}
Fix an integer $i, 1\le i \le s$. Let $E$ be an algebraic $r$-bundle over $F$ of type $\underline{a}_E^{(i)}=(a_1^{(i)},\ldots,a_r^{(i)}), a_1^{(i)}\geq\cdots\geq a_r^{(i)}$ with respect to $F^{(i)}$. If for some $t<r$,
\[
a_t^{(i)}-a_{t+1}^{(i)}\geq
\left\{
\begin{array}{ll}
1, & and~F^{(i)}~ \text{is in Case I}\\
2, & and ~F^{(i)}~ \text{is in Case II},
\end{array}
\right.\]
then there is a normal subsheaf $K\subseteq E$ of rank $t$ with the following properties: over the open set $V_E^{(i)}=q_1({q_2}^{-1}(U_E^{(i)}))\subseteq F$, the sheaf $K$ is a subbundle of $E$, which on the line $L\subseteq F$ given by $l\in U_E^{(i)}$ has the form\[
K|L\cong\oplus_{j=1}^{t}\mathcal{O}_L(a_j^{(i)}).
\]
\end{thm}
\begin{proof}
	After tensoring  with an appropriate line bundle, we may assume $a_t^{(i)}=0,a_{t+1}^{(i)}<0$.
\begin{enumerate}
	\item If $F^{(i)}$ is in Case I, then we have the natural projection $F\xrightarrow{q} F^{(i)}$. By the hypothesis, for every point  $l\in U_E^{(i)}$ \[
	E|q^{-1}(l)\cong\oplus_{j=1}^{r}\mathcal{O}_{q^{-1}(l)}(a_j^{(i)}).
	\]
	Then $q_{\ast}E$ is a coherent sheaf over $F^{(i)}$ which is locally free over $U_E^{(i)}$. The morphism $\phi:q^{\ast}q_{\ast}E\rightarrow E$ on  each $\widetilde{L}=q^{-1}(l)\cong L$ for an $l\in U_E^{(i)}$ is given by the evaluation of the section of $E|L$. Thus the image of $\phi|\widetilde{L}$ is the subbundle\[
	\oplus_{j=1}^{t}\mathcal{O}_L(a_j^{(i)})\subseteq E|L
	\]
	of rank $t$. Hence over the open set $q^{-1}(U_E^{(i)})$, $\phi$ is a morphism of constant rank $t$ and thus its image $Im \phi\subseteq E$ over $q^{-1}(U_E^{(i)})$ is a subbundle of rank $t$.
	
	Let $Q'=E/Im\phi$ and $T(Q')$ be the torsion subsheaf of $Q'$ and \[
	K=ker(E\rightarrow Q'/T(Q')).
	\]
	Because $\widetilde{Q}=Q'/T(Q')$ is a torsion-free sheaf, $K$ is a normal subsheaf of rank $t$. Over the open set $V_E^{(i)}=q^{-1}(U_E^{(i)})\subseteq F$ the sheaf $K$ is a subbundle of $E$, which on the line $L\subseteq F$ given by $l\in U_E^{(i)}$ has the form\[
	K|L\cong\oplus_{j=1}^{t}\mathcal{O}_L(a_j^{(i)}).
	\]
	\item If $F^{(i)}$ is in Case II, then we have the standard diagram
	\begin{align}
	\xymatrix{
		\widetilde{F^{(i)}}\ar[d]^{q_1}   \ar[r]^-{q_2} & F^{(i)}=F(d_1,\ldots,d_{i-1},d_i-1,d_i+1,d_{i+1},\ldots d_s)\\
		F=F(d_1,\ldots,d_s).
	}
	\end{align}
For every point  $l\in U_E^{(i)}$, we have\[
{q_1}^{\ast}E|{q_2}^{-1}(l)\cong E|L\cong\oplus_{j=1}^{r}\mathcal{O}_{L}(a_j^{(i)}).
\]
Then ${q_2}_{\ast}{q_1}^{\ast}E$ is a coherent sheaf over $F^{(i)}$ which is locally free over $U_E^{(i)}$. The morphism $\phi:{q_2}^{\ast}{q_2}_{\ast}{q_1}^{\ast}E\rightarrow {q_1}^{\ast}E$ on each $\widetilde{L}={q_2}^{-1}(l)\cong L$ for an $l\in U_E^{(i)}$ is given by the evaluation of the section of $E|L$. Thus the image of $\phi|\widetilde{L}$ is the subbundle\[
\oplus_{j=1}^{t}\mathcal{O}_L(a_j^{(i)})\subseteq E|L
\]
of rank $t$. Hence over the open set ${q_2}^{-1}(U_E^{(i)})$, $\phi$ is a morphism of constant rank $t$ and thus its image $Im \phi\subseteq {q_1}^{\ast}E$ over ${q_2}^{-1}(U_E^{(i)})$ is a subbundle of rank $t$.

Let $Q'={q_1}^{\ast}E/Im\phi$ and $T(Q')$ be the torsion subsheaf of $Q'$ and \[
\widetilde{K}=ker({q_1}^{\ast}E\rightarrow Q'/T(Q')).
\]
Because $\widetilde{Q}=Q'/T(Q')$ is a torsion-free sheaf, $\widetilde{K}$ is a normal subsheaf of rank $t$, and outside the singularity set $S(\widetilde{Q})$ of $\widetilde{Q}$, the sheaf $\widetilde{K}$ is a subbundle of ${q_1}^{\ast}E$, which on  each $\widetilde{L}={q_2}^{-1}(l)\cong L$ given by $l\in U_E^{(i)}$ has the form\[
\widetilde{K}|\widetilde{L}\cong\oplus_{j=1}^{t}\mathcal{O}_{\widetilde{L}}(a_j^{(i)}).
\]	
Let $X=\widetilde{F^{(i)}}\backslash S(\widetilde{Q})$. $X$ is open in $\widetilde{F^{(i)}}$ and contains ${q_2}^{-1}(U_E^{(i)})$. We have the following commutative diagram
\begin{align}\label{tt}
\xymatrix{
	&X  \ar@{^{(}->}[r]^{\zeta}\ar[d]_{f} & \widetilde{F^{(i)}}\ar[d]^{q_1}
	\\
	&Y=q_1(X) \ar@{^{(}->}[r]^-{\eta} & F
}
\end{align}
with a surjective submersion $f$ with connected fibres.

In order to apply the Descent Lemma to the subbundle $\widetilde{K}|X\subseteq f^{\ast}(E|Y)$, we need to show that \[
Hom(T_{X/Y},\mathcal{H}om(\widetilde{K}|X,\widetilde{Q}|X))=0.
\]
By the following Claim \ref{cl}, the hypothesis of the Descent Lemma is satisfied. Hence over the open set $Y\subseteq F$, we get a subbundle $K'\subseteq E|Y$
with\[
f^{\ast}K'=\widetilde{K}|X\subseteq f^{\ast}(E|Y).
\]
$K'$ can be extended to a normal subsheaf $K={q_1}_{\ast}\widetilde{K}\subseteq E$ on $F$:
To prove this, we need to consider the above diagram \ref{tt} again. From the diagram and Zariski's Main Theorem, we deduce that\[
f_{\ast}\mathcal{O}_X=\eta^{\ast}\eta_{\ast}f_{\ast}\mathcal{O}_X=\eta^{\ast}{q_1}_{\ast}\zeta_{\ast}\mathcal{O}_X=\eta^{\ast}{q_1}_{\ast}\mathcal{O}_{\widetilde{F^{(i)}}}=\eta^{\ast}\mathcal{O}_{F}=\mathcal{O}_{Y}.
\]
Next, we only need to prove $K|Y=K'$, i.e. $\eta_{\ast}K'={q_1}_{\ast}\widetilde{K}$.

Because $S(\widetilde{Q})$ is of codimension at least 2 and $\widetilde{K}$ is a normal sheaf, we have $\zeta_{\ast}(\widetilde{K}|X)=\widetilde{K}$. Thus\[
\eta_{\ast}K'=\eta_{\ast}(K'\otimes f_{\ast}\mathcal{O}_X)=\eta_{\ast}f_{\ast}f^{\ast}K'={q_1}_{\ast}\zeta_{\ast}(f^{\ast}K')={q_1}_{\ast}\zeta_{\ast}(\widetilde{K}|X)={q_1}_{\ast}\widetilde{K}.
\]
It's easy to see that over the open set $V_E^{(i)}=q_1({q_2}^{-1}(U_E^{(i)}))\subseteq F$, the sheaf $K$ is a subbundle of $E$, which on the line $L\subseteq F$ given by $l\in U_E^{(i)}$ has the form\[
K|L\cong\oplus_{j=1}^{t}\mathcal{O}_L(a_j^{(i)}).
\]
\end{enumerate}
\begin{claim}\label{cl}
	If $a_{t+1}^{(i)}<-1$, then\[
	Hom(T_{X/Y},\mathcal{H}om(\widetilde{K}|X,\widetilde{Q}|X))=0.
	\]
\end{claim}
In fact, it's equivalent to prove \[
H^{0}(X,\Omega_{\widetilde{F^{(i)}}/F}\otimes\widetilde{K}^{\vee}\otimes\widetilde{Q})=0.
\]
Since the codimension of $X\backslash{q_2}^{-1}(U_E^{(i)})$ in $X$ is at least $2$ and $\Omega_{\widetilde{F^{(i)}}/F}\otimes\widetilde{K}^{\vee}\otimes\widetilde{Q}$ is torsion free, the restriction\[
H^{0}(X,\Omega_{\widetilde{F^{(i)}}/F}\otimes\widetilde{K}^{\vee}\otimes\widetilde{Q})\rightarrow H^{0}({q_2}^{-1}(U_E^{(i)}),\Omega_{\widetilde{F^{(i)}}/F}\otimes\widetilde{K}^{\vee}\otimes\widetilde{Q})
\]
is injective, it suffices to show that $\Omega_{\widetilde{F^{(i)}}/F}\otimes\widetilde{K}^{\vee}\otimes\widetilde{Q}$ has no section over ${q_2}^{-1}(U_E^{(i)})$.

Let $l\in U_E^{(i)}$ and $\widetilde{L}={q_2}^{-1}(l)\cong L$. By the previous assertion, \[
\widetilde{K}^{\vee}|\widetilde{L}\cong\oplus_{j=1}^{t}\mathcal{O}_{\widetilde{L}}(-a_j^{(i)}),~\widetilde{Q}|\widetilde{L}\cong\oplus_{j=t+1}^{r}\mathcal{O}_{\widetilde{L}}(a_j^{(i)}),
\]	
and by Lemma \ref{rt}, we have \[
\Omega_{\widetilde{F^{(i)}}/F}|\widetilde{L}=\mathcal{O}_{\widetilde{L}}(1)^{\oplus d_{i+1}-d_{i-1}-2}.
\]
Thus \[
H^{0}(\widetilde{L},\Omega_{\widetilde{F^{(i)}}/F}\otimes\widetilde{K}^{\vee}\otimes\widetilde{Q}|\widetilde{L})=0,~\text{if}~ a_{t+1}^{(i)}<-1.
\]
Then every section of $(\Omega_{\widetilde{F^{(i)}}/F}\otimes\widetilde{K}^{\vee}\otimes\widetilde{Q}|X)$ is zero over ${q_2}^{-1}(U_E^{(i)})$ and hence over $X$.
\end{proof}
The above theorem has far reaching consequences. We give first a series of immediate deduction.
\begin{cor}
Fix an integer $i, 1\le i \le s$. Let $E$ be a uniform $r$-bundle with respect to $F^{(i)}$ of type\[
\underline{a}_E^{(i)}=(a_1^{(i)},\ldots,a_r^{(i)}), a_1^{(i)}\geq\cdots\geq a_r^{(i)}
\]
If for some $t<r$,
 \[
 a_t^{(i)}-a_{t+1}^{(i)}\geq
 \left\{
 \begin{array}{ll}
 1, & and~F^{(i)}~ \text{is in Case I}\\
 2, & and ~F^{(i)}~ \text{is in Case II},
 \end{array}
 \right.\]
then we can write $E$ as an extension of uniform bundles with respect to $F^{(i)}$.
\end{cor}
\begin{proof}
By the above Theorem \ref{main}, there is a uniform bundle $K\subseteq E$ of type $\underline{a}_K^{(i)}=(a_1^{(i)},\ldots,a_t^{(i)})$ with respect to $F^{(i)}$. Then the quotient bundle $Q=E/K$ is uniform of type $(a_{t+1}^{(i)},\ldots,a_r^{(i)})$ with respect to $F^{(i)}$. We have the following exact sequence\[
0\rightarrow K\rightarrow E\rightarrow Q\rightarrow 0.
\]
\end{proof}

Let $\mathcal{F}$ be a torsion free coherent sheaf of rank $r$ over $F$.
Since the singularity set $S(\mathcal{F})$ of $\mathcal{F}$ has codimension at least $2$, there is some integer $i(1\le i \le s)$ and lines $L\subseteq F$ given by $l\in F^{(i)}$ which do not meet $S(\mathcal{F})$. If
\[\mathcal{F}|L\cong \mathcal{O}_{L}(a_1^{(i)})\oplus\cdots\oplus \mathcal{O}_{L}(a_r^{(i)}).\]
Let \[c_1^{(i)}(\mathcal{F})=a_1^{(i)}+\cdots+a_r^{(i)},\] which is independent of the choice of $L$.
We set
\[\mu^{(i)}(\mathcal{F})=\frac{c_1^{(i)}(\mathcal{F})}{\text{rk}(\mathcal{F})}.\]

\begin{definition}\label{i}
A torsion free coherent sheaf $\mathcal{E}$ over $F$ is i-semistable if for every coherent subsheaf $ \mathcal{F}\subseteq \mathcal{E}$ with $0< \text{rk}(\mathcal{F})< \text{rk}(\mathcal{E})$, we have
\[\mu^{(i)}(\mathcal{F})\le \mu^{(i)}(\mathcal{E}).\]
\end{definition}

\begin{cor}\label{gap}
Fix an integer $i, 1\le i \le s$. For an i-semistable $r$-bundle $E$ over $F$ of type $\underline{a}_E^{(i)}=(a_1^{(i)},\ldots,a_r^{(i)}), a_1^{(i)}\geq\cdots\geq a_r^{(i)}$ with respect to $F^{(i)}$, we have \[
a_j^{(i)}-a_{j+1}^{(i)}\leq 1~~ \text{for all}~ j=1,\ldots, r-1.
\]
In particular, if $F^{(i)}$ is in Case I, then $a_j^{(i)}$'s are constants for all $1\leq j\leq r$ .
\end{cor}
\begin{proof}
 If for the fixed $i$, $E$ is of type $\underline{a}_E^{(i)}=(a_1^{(i)},\ldots,a_r^{(i)})$ with $a_t^{(i)}-a_{t+1}^{(i)}\geq 2$ for some $t<r$, then by Theorem \ref{main} we can find a normal sheaf $K\subseteq E$ which is of the form
 \[
 K|L\cong\oplus_{j=1}^{t}\mathcal{O}_L(a_j^{(i)})
 \]
 over the line $L\subseteq F$ given by $l\in U_E^{(i)}$. Then we have $\mu^{(i)}(E)<\mu^{(i)}(K)$, hence $E$ is not $i$-semistable.

 In particular, if the $i$-th irreducible component of the manifold of lines $F^{(i)}$ is in Case I and there is some $t<r$ such that $a_t^{(i)}\neq a_{t+1}^{(i)}$, then we could find a normal sheaf $K\subseteq E$ such that $\mu^{(i)}(E)<\mu^{(i)}(K)$, hence $E$ is not $i$-semistable.
\end{proof}

\begin{cor}\label{x}
If $E$ is a strongly uniform $i$-semistable $(1\le i\le n-1)$ $r$-bundle over the complete flag variety $F$, then $E$ splits as a direct sum of line bundles. In addition $E|L\cong \mathcal{O}_L(a)^{\oplus r}$ for every line $L\subseteq F$, where $a\in \mathbb{Z}$.
\end{cor}
\begin{proof}
By Corollary \ref{gap}, we get $E|L=\mathcal{O}_L(a)^{\oplus r}$ for every line $L$ in $F$. After tensoring with appropriate line bundles, we can assume $E|L=\mathcal{O}_{L}^{\oplus r}$ for every line $L$. So $E$ is trivial by Proposition \ref{y}.
\end{proof}


\section*{Acknowledgements}
All the authors would like to show our great appreciation to the anonymous reviewer for pointing out a fatal error in one of our propositions and providing many useful suggestions in the original version.
\bibliography{ref}

\begin{thebibliography}{10}

\bibitem{ref12}
V.~Ancona, T.~Peternell, and J.~Wi{\'s}niewski.
\newblock Fano bundles and splitting theorems on projective spaces and
  quadrics.
\newblock {\em Pacific J. Math.}, 163(1):17--42, 1994.

\bibitem{ref17}
E.~Ballico.
\newblock Uniform vector bundles on quadrics.
\newblock {\em Ann. Univ. Ferrara Sez. VII (N.S.)}, 27(1):135--146, 1982.

\bibitem{ref8}
E.~Ballico.
\newblock Uniform vector bundles of rank $(n+ 1) $ on $ \mathbb{P}^{n} $.
\newblock {\em Tsukuba J. Math.}, 7(2):215--226, 1983.

\bibitem{ref21}
W.~Barth.
\newblock Some properties of stable rank-2 vector bundles on $\mathbb{P}_n$.
\newblock {\em Math. Ann.}, 226:125--150, 1977.

\bibitem{ref6}
L.~Ein.
\newblock Stable vector bundles on projective spaces in char $p> 0$.
\newblock {\em Math. Ann.}, 254(1):53--72, 1980.

\bibitem{ref4}
G.~Elencwajg.
\newblock Les fibr\'es uniformes de rang 3 sur $\mathbb{P}^2(\mathbb{C})$ sont
  homog\'enes.
\newblock {\em Math.Ann}, 231:217--227, 1978.

\bibitem{ref5}
G.~Elencwajg, A.~Hirschowitz, and M.~Schneider.
\newblock Les fibr\'es uniformes de rang n sur $\mathbb{P}^n(\mathbb{C})$ sont
  ceux qu'on croit.
\newblock {\em Progr. Math.}, 7:37--63, 1980.

\bibitem{ref7}
P.~Ellia.
\newblock Sur les fibr{\'e}s uniformes de rang $(n+1)$ sur $\mathbb{P}^{n} $.
\newblock {\em M{\'e}m. Soc. Math. France (N.S.)}, 7:1--60, 1982.

\bibitem{ref19}
O.~Forster, A.~Hirschowitz, and M.~Schneider.
\newblock Type de scindage g{\'e}n{\'e}ralise pour les fibr{\'e}s stables.
\newblock {\em Progr. Math.}, (7):65--81, 1980.

\bibitem{ref15}
M.~Guyot.
\newblock Caract{\'e}risation par l'uniformit{\'e} des fibr{\'e}s universels
  sur la grassmanienne.
\newblock {\em Math. Ann.}, 270(1):47--62, 1985.

\bibitem{ref23}
R.~Hartshorne.
\newblock Stable reflexive sheaves.
\newblock {\em Math. Ann.}, 254:121--176, 1980.

\bibitem{ref24}
D.~Huybrechts and M.~Lehn.
\newblock {\em The geometry of moduli spaces of sheaves, second ed.}
\newblock Cambridge University Press, 2010.
\newblock xviii+325 pp.

\bibitem{ref22}
J-M. Hwang.
\newblock Geometry of minimal rational curves on fano manifolds.
\newblock {\em School on Vanishing Theorems and Effective Results in Algebraic
  Geometry (Trieste, 2000), ICTP Lect. Notes, vol. 6, Abdus Salam Int. Cent.
  Theoret. Phys., Trieste, 2001}, pages 335--393, 2001.

\bibitem{ref18}
Y.~Kachi and E.~Sato.
\newblock Segre’s reflexivity and an inductive characterization of
  hyperquadrics.
\newblock {\em Mem. Amer. Math. Soc.}, 160(763), 2002.

\bibitem{ref11}
N.~Katz.
\newblock Nilpotent connections and the monodromy theorem: Applications of a
  result of turrittin.
\newblock {\em Inst. Hautes {\'E}tudes Sci. Publ. Math.}, 39:175--232, 1970.

\bibitem{ref10}
H.~Lange.
\newblock On stable and uniform rank-$2$ vector bundles on $\mathbb{P}^2$ in
  characteristic p.
\newblock {\em Manuscripta Math.}, 29(1):11--28, 1979.

\bibitem{ref25}
A.~Langer.
\newblock Semistable sheaves in positive characteristic.
\newblock {\em Ann. Math.}, 159:251--276, 2004.

\bibitem{ref13}
R.~Lazarsfeld.
\newblock {\em Positivity in Algebraic Geometry II: Positivity for Vector
  Bundles, and Multiplier Ideals}, volume~49.
\newblock Springer-Verlag, Berlin, 2004.
\newblock xviii+385 pp.

\bibitem{ref9}
R.~Mu$\tilde{n}$oz, G.~Occhetta, and L.~Sol{\'a} Conde.
\newblock Uniform vector bundles on fano manifolds and applications.
\newblock {\em J. Reine Angew. Math. (Crelles Journal)}, 664:141--162, 2012.

\bibitem{ref14}
C.~Okonek, M.~Schneider, and H.~Spindler.
\newblock {\em Vector bundles on complex projective spaces}.
\newblock Birkh$\ddot{a}$user/Springer Basel AG, Basel, 2011.
\newblock viii+239 pp.

\bibitem{ref3}
E.~Sato.
\newblock Uniform vector bundles on a projective space.
\newblock {\em J. Math. Soc. Japan}, 28(1):123--132, 1976.

\bibitem{ref1}
R.~L.~E. Schwarzenberger.
\newblock Vector bundles on the projective plane.
\newblock {\em Proc. London Math. Soc.}, 11(3):623--640, 1961.

\bibitem{ref16}
H.~Tango.
\newblock On morphisms from projective space $\mathbb{P}_n$ to the {G}rassmann
  variety {G}r(n, d).
\newblock {\em J. Math. Kyoto Univ.}, 16(1):201--207, 1976.

\bibitem{ref2}
A.~{Van de Ven}.
\newblock On uniform vector bundles.
\newblock {\em Math. Ann.}, 195:245--248, 1972.

\end{thebibliography}

\end{document}